\documentclass[12pt]{article}
\usepackage{amssymb, amsmath, amsthm}

\newcommand{\R}{\mathbb R}
\newcommand{\N}{\mathbb N}
\newcommand{\Z}{\mathbb Z}

\newtheorem{theorem}{Theorem}
\newtheorem{lemma}{Lemma}

\begin{document}
\title{Conjugacy Growth in Polycyclic Groups}
\author{M. Hull}
\maketitle
\begin{abstract}
In this paper, we consider the conjugacy growth function of a group, which counts the number of conjugacy classes which intersect a ball of radius $n$ centered at the identity. We prove that in the case of virtually polycyclic groups, this function is either exponential or polynomially bounded, and is polynomially bounded exactly when the group is virtually nilpotent. The proof is fairly short, and makes use of the fact that any polycyclic group has a subgroup of finite index which can be embedded as a lattice in a Lie group, as well as  exponential radical of Lie groups and Dirichlet's approximation theorem.
\end{abstract}

\section{Introduction}
Given a finitely generated group $G$, along with some fixed finite generating set $S$, let $B(n)$ denote the ball of radius $n$ centered at the identity with respect to the standard word metric induced by $S$, and denote the corresponding norm function by $\vert\cdot\vert_G$.  We denote the {\em conjugacy growth function} of $G$ by $\gamma_G^c(n)$, defined as 
\[
\gamma_G^c(n)=\#\text{conjugacy classes } C \text{ such that } C\cap B(n)\neq\emptyset
\]

This function was introduced  to count closed geodesics up to free homotopy on complete Riemannian manifolds \cite{IB}; a more complete history of this function can be found in \cite{GS}. We will also use $\gamma_G(n)$ to denote the standard growth function, that is $\gamma_G(n)=\#B(n)$. We consider these functions up to the following equivalence; $f\preceq g$ if there exist constants $C$ and $D$ such that $f(n)\leq Cg(Dn)$  $\forall n\in\N$, and $f\sim g$ if $f\preceq g$ and $g\preceq f$. We say that $f$ is {\em polynomially bounded} if $f\preceq n^d$ for some $d\in\N$ and $f$ is {\em exponential} if $f\sim 2^n$. Note that unlike the standard growth function, polynomially bounded does not imply polynomial in the case of conjugacy growth functions \cite[Example 2.4]{GS}. In \cite{GS}, it was conjectured that every amenable group of exponential growth has exponential conjugacy growth, and an amenable growth of polynomial conjugacy growth is virtually nilpotent. Our main result gives a proof of this conjecture in the case of virtually polycyclic groups.
\begin{theorem}
Let $P$ be a virtually polycyclic group. Then the conjugacy growth function of $P$ is either polynomially bounded or exponential. Furthermore, the conjugacy growth function of $P$ is polynomially bounded if and only if $P$ is virtually nilpotent. 
\end{theorem}
This theorem is an extension of a well known result for the ordinary growth function proved by Wolf in \cite{W}.
The author would like to mention that he is aware of an independent proof  of this result by Emmanuel Breuillard and Yves De Cornulier for the more general case of virtually solvable groups. The author would also like to thank Denis Osin for providing some important ideas in the proof, and Curt Kent for many helpful conversations.

\section{Proofs}
Recall that if $P$ is virtually nilpotent, then $\gamma_P$ is polynomial \cite{B}. Since clearly $\gamma_P^c\leq\gamma_P$, we have that $\gamma_P^c$ is polynomially bounded. Thus, it remains to show that if $P$ is not virtually nilpotent, then $\gamma_P^c$ is exponential. Since for any group $G$, $\gamma_G^c\leq\gamma_G\preceq2^n$, we only need show that $2^n\preceq\gamma_P^c$. We begin with the following straightforward lemma.
\begin{lemma}
Let $H$ be a subgroup of finite index in $G$. Then $\gamma_H^c\preceq\gamma_G^c$.
\end{lemma}
\begin{proof}
Let $S$ and $T$ be finite generating sets for $G$ and $H$ respectively. Let $k=\max\{\vert t\vert_G: t\in T\}$. Let $g_1,...g_m$ be a set of left coset representatives for $H$. Suppose $h_1$, and $h_2$ are conjugate in $G$ but not in $H$. Then for some $g_i$ and $h\in H$, $g_ihh_1h^{-1}g_i^{-1}=h_2$. Thus, $g_i$ combines the conjugacy classes of $h_1$ and $h_2$;  it follows that if $h\in H$ and $\vert h\vert_H\leq n$, then $\vert h\vert_G\leq kn$ and the  conjugacy class of $h$ in $G$ splits into at most $m+1$ conjugacy classes in $H$. Thus, $\gamma_H^c(n)\leq(m+1)\gamma_G^c(kn)$.
\end{proof}
Given a connected Lie group $\Gamma$ generated by a compact, symmetric neighborhood of the identity $\Omega$, we define $\vert g\vert_{\Gamma}=\inf\{n:g\in\Omega^n=\Omega . . . \Omega\}$. Since this is independent (up to quasi-isometry) of the choice of $\Omega$, we refer to this as the {\em canonical norm} on $\Gamma$. This norm induces a metric in the usual way, that is $dist(g,h)=\vert g^{-1}h\vert_{\Gamma}$. Note that this metric is quasi-isometric to the metric induced by any left-invariant Riemannian structure on $\Gamma$. Also, if $\Delta$ is a connected Lie subgroup of $\Gamma$ with canonical norm $\vert\cdot\vert_{\Delta}$, then there exists a universal constant $D$ such that $\vert h\vert_{\Gamma}\leq D\vert h\vert_{\Delta}$ for all $h\in\Delta$. 

We define the {\em exponential radical} of $\Gamma$, denoted $Exp(\Gamma)$ to be the set of all $g\in\Gamma$ satisfying
\[
\frac{1}{c}\log(\vert n\vert+1)-\epsilon\leq \vert g^n\vert_{\Gamma}\leq c\log(\vert n\vert+1)+\epsilon
\]

for some constants $c, \epsilon$ and for all $n\in\Z$. Such $g$ are called {\em strictly exponentially distorted}. For convenience, we consider $1$ to be strictly exponentially distorted. A construction of the exponential radical is given in \cite{O}.

We will also make use of Dirichlet's approximation theorem, which can be found in \cite[Theorem 200]{HW}.
\begin{theorem}
For any real numbers $c_1,..., c_k$, and any integer $m>0$, there exist integers $q, p_1,...,p_k$ such that $\vert q\vert\le m^k$ and $\vert qc_i-p_i\vert<\frac{1}{m}$.
\end{theorem}

Now we are ready to prove the main theorem.

\begin{proof}
Let $P$ be a group which is virtually polycyclic but not virtually nilpotent. Then $P$ has a subgroup $P_0$ of finite index which can be embedded as a lattice into a connected, simply-connected, solvable Lie group $\Gamma$ \cite[Theorem 4.28]{R}. Furthermore, since every lattice in a connected, solvable Lie group is uniform \cite[Theorem 3.1]{R}, this embedding is a quasi-isometry; this is a consequence of the Milnor-\v{S}varc Lemma (see, for example, \cite{S}). In view of Lemma 1, it suffices to show that $\gamma_{P_0}^c$ is at least exponential.

Let $\Delta$ be the maximal connected nilpotent normal subgroup of $\Gamma$, and $Z(\Delta)$ the center of $\Delta$. We identify $P_0$ with its image in $\Gamma$. Then $P_0\cap\Delta$ is a lattice in $\Delta$ \cite[Corollary 3.5]{R}, and furthermore $P_0\cap Z(\Delta)$ is a lattice in $Z(\Delta)$ \cite[Proposition 2.17]{R}. Since $Z(\Delta)$ is invariant under any automorphism of $\Delta$, $Z(\Delta)\lhd\Gamma$, and thus $P_0\cap Z(\Delta)\lhd P_0$. Note that since $P$ is not virtually nilpotent, $P_0$ is also not virtually nilpotent.  It is proved in \cite{O} that if $\Gamma$ contains a lattice which is not virtually nilpotent, then $Exp(\Gamma)\neq\{1\}$, and in particular $Exp(\Gamma)\cap Z(\Delta)\neq\{1\}$. Let $v\in Exp(\Gamma)\cap Z(\Delta)$, $v\ne 1$. Since $Z(\Delta)$ is a connected, simply-connected, abelian Lie subgroup \cite[Chap. III, Sec. 9, Proposition 15]{NB}, we have $Z(\Delta)\cong\R^k$ for some $k$. Thus, we can identify  $P_0\cap Z(\Delta)$ with $\Z^k$ inside of $\R^k$. 

Fix $n\in\N$ sufficiently large. Applying Dirichlet's approximation theorem to the coordinates of $v$ with $m=2^{\frac{n}{k+1}}$,  we can find $q\in\Z$ and $u\in\Z^k$ such that $\vert\vert qv-u\vert\vert\leq\frac{\sqrt{k}}{m}$ where $\vert\vert\cdot\vert\vert$ denotes the standard Euclidean norm. Now, let $t\in\Z$, $1\leq t\leq m$. Then $\vert\vert tqv-tu\vert\vert\leq\sqrt{k}$. Since  the Euclidean norm on $\R^k$ is quasi-isometric to a canonical norm on $Z(\Delta)$ which is a connected Lie subgroup of $\Gamma$, there is a constant $C_0$ such that $\vert tqv-tu\vert_{\Gamma}\leq C_0\vert\vert tqv-tu\vert\vert\leq C_0\sqrt{k}$. Since $v$ is strictly exponentially distorted in $\Gamma$ and $\vert tq\vert\leq m^{k+1}=2^n$, $\vert tqv\vert_{\Gamma}\leq C_1n$ for some constant $C_1$. Thus, there exists a constant $C_2$ such that $\vert tu\vert_{\Gamma}\leq C_2n$. Finally, $P_0$ is embedded quasi-isometrically in $\Gamma$, so there is a constant $C$ such that $\vert tu\vert_{P_0}\leq Cn$. 

Now suppose $w, tw\in\Z^k$ and $gwg^{-1}=tw$ for some $g\in P_0$. Since $\Z^k\lhd P_0$, conjugation by $g$ is an automorphism of $\Z^k$; thus, we can identify this automorphism with a matrix $\phi_g\in GL(k, \Z)$. By the rational root theorem, the only rational eigenvalues of $\phi_g$ are $1$ and $-1$; it follows that $t=\pm 1$. Therefore, $\{tu: 1\leq t\leq m\}$ are all representatives of distinct conjugacy classes in $P_0$. Thus, we have found $2^{\frac{n}{k+1}}$ non-conjugate elements in the ball of radius $Cn$ in $P_0$, so $2^n\preceq\gamma_{P_0}^c$.
\end{proof}

M. Hull \\
Mathematics Department \\
1326 Stevenson Center \\
Vanderbilt University \\
Nashville, TN 37240 \\
email: michael.b.hull@vanderbilt.edu

\end{document}